\documentclass[11pt]{article}

\usepackage{amsmath, amsfonts, amssymb, amsthm, mathtools, amscd}
\usepackage{mathrsfs} 
\usepackage{dsfont}
\usepackage[mathscr]{eucal}
\usepackage{graphicx}
\usepackage{csquotes}
\usepackage{enumitem}

\usepackage[T1]{fontenc}
\usepackage{newtxtext, newtxmath} 

\usepackage[english]{babel}
\usepackage{microtype}
\usepackage{setspace}
\usepackage[a4paper, margin=1.25in]{geometry}
\usepackage{titling}

\pretitle{\begin{center}\LARGE\bfseries}
	\posttitle{\par\end{center}\vskip 0.5em}
\preauthor{\begin{center}\large}
	\postauthor{\end{center}}
\predate{\begin{center}\normalsize}
	\postdate{\end{center}}

\usepackage{titlesec}

\titleformat{\section}{\large\bfseries}{\thesection}{1em}{}
\titleformat{\subsection}{\normalsize\bfseries}{\thesubsection}{1em}{}

\newcommand{\N}{{\mathbb N}}
\newcommand{\Z}{{\mathbb Z}}
\newcommand{\Q}{{\mathbb Q}}
\newcommand{\R}{{\mathbb R}}
\newcommand{\abs}[1]{\left|#1\right|}
\newcommand{\norm}[1]{\left\lVert #1 \right\rVert}

\newcommand{\floor}[1]{\left\lfloor#1\right\rfloor}

\newcommand{\her}[1]{\left\langle #1 \right\rangle}

\renewcommand{\le}{\leqslant}
\renewcommand{\ge}{\geqslant}
\DeclareMathOperator*{\sumflat}{\sideset{}{^{\flat}}\sum}

\usepackage{hyperref} 
\hypersetup{
	colorlinks=true,      
	linkcolor=blue,        
	urlcolor=blue,         
	citecolor=blue        
}

\usepackage[capitalize]{cleveref}

\usepackage{aliascnt}

\newtheorem{thm}{Theorem}
\newtheorem*{thm*}{Theorem}

\newcommand{\newaliasedtheorem}[3]{%
	\newaliascnt{#1}{thm}%
	\newtheorem{#1}[#1]{#3}%
	\aliascntresetthe{#1}%
	\crefname{#1}{#2}{#2s}%
	\Crefname{#1}{#3}{#3s}%
}

\newaliasedtheorem{lem}{lemma}{Lemma}
\newaliasedtheorem{clm}{claim}{Claim}
\newaliasedtheorem{prop}{proposition}{Proposition}
\newaliasedtheorem{cor}{corollary}{Corollary}
\newaliasedtheorem{axm}{axiom}{Axiom}
\newaliasedtheorem{fact}{fact}{Fact}
\newaliasedtheorem{conj}{conjecture}{Conjecture}

\theoremstyle{remark}
\newaliasedtheorem{rem}{remark}{Remark}
\newaliasedtheorem{rmk}{remark}{Remark}
\newaliasedtheorem{note}{note}{Note}
\newaliasedtheorem{qs}{question}{Question}
\newaliasedtheorem{obs}{observation}{Observation}

\theoremstyle{definition}
\newaliasedtheorem{defn}{definition}{Definition}

\title{Lower Bounds for Localised Moments of Automorphic $L$-functions}
\author{Deep Thakur}
\date{}

\begin{document}
	\maketitle	
	\bgroup
	\let\thefootnote\relax\footnotetext{The Institute of Mathematical Sciences, A CI of Homi Bhabha National Institute, 
		CIT Campus, Taramani, Chennai 600 113, India. Email: deepthakur@imsc.res.in}
	\egroup
	
	\begin{abstract}
		In this article, we describe a method to obtain lower bounds for the $2k$-th moment of values of $L$-functions localised at the central point for all real $k > 0$. In particular, for the families of quadratic twists of self-dual irreducible unitary cuspidal automorphic representations of $\operatorname{GL}_D(\mathbb{A}_{\Q})$ with $D\in\{2,3,4\}$, our results are new in the range $0<k<\tfrac12$ when $D\in\{2,3\}$, and new for all $k>0$ when $D=4$.
	\end{abstract}

	\section{INTRODUCTION}

	A classical theme in analytic number theory is to study the moments of the Riemann zeta function on the critical line
	$$
	\int_{T}^{2T} \abs{\zeta(\tfrac{1}{2}+it)}^{2k}dt,
	$$
	where $k$ is a positive real number. A folklore conjecture states that the $2k$-th moment should be asymptotic to $C_k T(\log T)^{k^2}$ for some positive constant $C_k$. For $k=1$ and $2$ this is known due to the classical work of Hardy--Littlewood \cite{HardyLittlewood1916} and Ingham \cite{Ingham1927b} respectively, and these still remain 
	the only moments of the Riemann zeta function where asymptotics are known 
	unconditionally. Pioneering work by Ramachandra \cite{Ramachandra1978,Ramachandra1980a,Ramachandra1980b} 
	and Heath-Brown \cite{HeathBrown1981} established lower bounds 
	of the correct order of 
	magnitude unconditionally when 
	$k> 0$ is rational, and assuming the Riemann hypothesis when $k> 0$ is real. 
	Heap--Soundararajan \cite{HeapSoundararajan2022} introduced a novel method which extended these lower bounds unconditionally to all 
	real $k> 0$ (see also \cite{RadziwillSoundararajan2013}). In recent joint work with Gun and Kumar \cite{GKT} we obtained lower bounds for the $2k$-th moment of a broad class of $L$-functions for all real $k> 0$, which includes in particular the Dedekind zeta functions and automorphic $L$-functions.
	
	Analogously, given a family $\mathcal{F}$ of $L$-functions one studies the moments of the central values of these $L$-functions:
	$$
	\sum_{f\in\mathcal{F}(X)} \abs{L(\tfrac{1}{2},f)}^{2k} \text{ as } X\to\infty,
	$$
	where $\mathcal{F}(X)=\{f\in\mathcal{F}:\ C(f)\le X\}$ and $C(f)$ denotes the analytic conductor of $L(s,f)$ at the central point.
	Due to the works of Keating--Snaith \cite{Keating2000}, Diaconu--Goldfeld--Hoffstein \cite{DGH} and Conrey--Farmer--Keating--Rubinstein--Snaith \cite{CFKRS} there are now well established conjectures for the asymptotic formulas for these $2k$-th moments. While there are many examples of families of $L$-functions where the asymptotic formulas are known for small integer values of $2k$, in general the moment conjectures seem formidable. Rudnick--Soundararajan \cite{RudnickSoundararajan2005,RudnickSoundararajan2006} introduced a simple method which establishes lower bounds of the correct order for moments of $L$-functions arising from orthogonal and symplectic families when $k\ge \frac12$ is rational. This method was subsequently refined by Radziwi{\l}{\l}--Soundararajan \cite{RadziwillSoundararajan2013} to cover all real $k\ge \frac12$. Heap--Soundararajan \cite{HeapSoundararajan2022} extended these lower bounds to the range $0< k < \frac12$ (see also \cite{Radziwill2024}), however their method crucially relies on sharp upper bounds for the twisted second moments for the family of $L$-functions:
	$$
	\sum_{f\in\mathcal{F}(X)} \abs{L(\tfrac{1}{2},f)A_f(\tfrac{1}{2})}^2,
	$$
	where $A_f(s)$ are suitably chosen short Dirichlet polynomials. 
	
	Let $D\in\{2,3,4\}$ and let $\pi$ be a self-dual irreducible unitary cuspidal
	automorphic representation of $\operatorname{GL}_D(\mathbb{A}_{\Q})$, and 
	$$
	\mathcal{F}_{\pi} = \{\pi\otimes\chi_{8d} : \mu^2(2d)=1\}
	$$
	be the family of quadratic twists of $\pi$. One is interested in understanding the moments of $L$-functions in the family $\mathcal{F}_{\pi}$:
	$$
	\mathcal{M}_k(\pi,X)=\sum_{\substack{1 \le d \le X \\ \mu^2(2d)=1}}
	\abs{L(\tfrac{1}{2},\pi\otimes\chi_{8d})}^{2k}.
	$$
	Set $\epsilon = \epsilon(\pi) \in \{-1,+1\}$ so that 
	$$ 
	\epsilon(\pi) = 
	\begin{cases} 
		-1 & \text{if }
		L(s,\wedge^{2}\pi) \text{ has a pole at } s=1, \\[10pt]
		+1 & \text{otherwise.}
	\end{cases}
	$$
	We have the following conjectured asymptotic formula \cite{Keating2000,DGH,CFKRS}.
	\begin{conj}\label{conj:moments}
		For all $k>0$ there exists a constant $C_{k}(\pi)$ such that
		$$
		\mathcal{M}_k(\pi,X)\sim C_{k}(\pi) X(\log X)^{k(2k+\epsilon)} \text{ as } X\to\infty.
		$$ 
	\end{conj}
	When $\pi$ corresponds to a holomorphic Hecke eigenform for $\operatorname{SL}_2(\Z)$, or to a self-dual Hecke--Maass cusp form for $\operatorname{SL}_3(\Z)$, the twisted first moments are known (see \cite{Shen, HuaHuang2023a}) and hence the lower bounds for $\mathcal{M}_k(\pi,X)$ of
	the conjectured order of magnitude are known for $k\ge\tfrac12$ (see \cite{GaoZhao, HuaHuang2023b}). The corresponding twisted second moment estimates, and hence the lower bounds of the conjectured order for $\mathcal{M}_k(\pi,X)$ in the range $0<k<\tfrac12$~, remain difficult open problems.
	
	In this article, in order to circumvent these twisted moment estimates, we consider the moments of $L$-functions localised at the central point:
	\begin{equation}\label{eq:moment}
		\mathcal{I}_k(\pi,X,\phi) = \sum_{\substack{1 \le d \le X \\ \mu^2(2d)=1}}\int_{-\infty}^{\infty}
		\abs{L(\tfrac{1}{2}+it,\pi\otimes\chi_{8d})}^{2k}\abs{W_{X}(t)}^{2}\,dt,
	\end{equation}
	where $W_X(t)=(\log X)^{1/2}\,\widehat{\phi}(t\log X)$ for all $t\in\R$ and $\widehat{\phi}(\xi)=\int_{-\infty}^{\infty}\phi(x)e^{-ix\xi}dx$ for all $\xi\in\R$, for a fixed nonzero function $\phi\in C_c^\infty(\R)$. Thus
	$$
	\int_{-\infty}^{\infty} \abs{W_X(t)}^2dt\asymp_\phi 1 \text{ and } \abs{W_X(t)}\ll_{\phi,A} \frac{(\log X)^{1/2}}{(1+\abs{t}\log X)^{A}} \text{ for every } A>0,
	$$
	so that \eqref{eq:moment} is a moment localised at the central point. We take $\phi=\phi_\delta$, where for $\delta>0$ we set
	$$ 
	\phi_\delta(x) = 
	\begin{cases} 
		\dfrac{1}{\delta}\exp\left(-\dfrac{1}{1-(x/\delta)^2}\right) & \text{if } \abs{x} < \delta, \\[10pt]
		0 & \text{if } \abs{x} \ge \delta.
	\end{cases}
	$$
	
	The localisation allows us to combine the method of Heap--Soundararajan \cite{HeapSoundararajan2022} with the convexity argument of Heath-Brown \cite{HeathBrown1981,HeathBrown2010}. In particular, our lower bound method covers the range $0<k<\frac12$, which is not accessible to the method of Heap--Soundararajan \cite{HeapSoundararajan2022} without twisted second moment estimates. We have the following theorem.
	
	\begin{thm}\label{thm:mainthm}
		Let $D\in\{2,3,4\}$ and let $\pi$ be a self-dual irreducible unitary cuspidal
		automorphic representation of $\operatorname{GL}_D(\mathbb{A}_{\Q})$, and let $\chi_{8d}$ denote the primitive quadratic
		character of the fundamental discriminant $8d$. Then for every real $k>0$ there is a $\delta=\delta(\pi,k)>0$ such that, for
		all sufficiently large $X$,
		$$
		\mathcal{I}_k(\pi,X,\phi_\delta) \gg_{\pi,k}\ X(\log X)^{\,k(2k+\epsilon)}.
		$$
	\end{thm}
	
	\begin{rem}
		In view of \Cref{conj:moments} we see that our lower bounds for the quantity $\mathcal{I}_k(\pi,X,\phi_\delta)$ are of the expected order of magnitude.
	\end{rem}
	
	\begin{rem}
		If $\omega_{\pi}$ denotes the central character of $\pi$, then $\wedge^{2}\pi=\omega_\pi$ when $D=2$ and
		$\wedge^{2}\pi\cong\tilde\pi\otimes\omega_\pi$ when $D=3$ \cite[Introduction]{Kim}. Thus, for $D\in\{2,3\}$,
		we have $\epsilon=-1$ if and only if $D=2$ and $\omega_\pi$ is trivial.
	\end{rem}
	
	\begin{rem}
		The method presented in the proof of \Cref{thm:mainthm} can be adapted to a broad class of families of $L$-functions where analogues of results in \Cref{sec:Preliminaries} are available. 
	\end{rem}
	
	Throughout this article, the implicit constants in the $\ll, \gg, \asymp,$ and $O(\cdot)$ notations may depend on the automorphic representation $\pi$ and the real number $k$. Any dependence on further parameters will be explicitly indicated using subscripts (e.g., $\ll_\varepsilon$). 
	
	\section{SETUP AND PLAN OF PROOF}
	
	Let $k>0$ be fixed. We denote by $\lambda(n)$ the Dirichlet coefficients of $L(s,\pi)$, $q$ the conductor of $\pi$ and $\alpha_1(p),\,\ldots,\alpha_D(p)$ the local parameters of $\pi$ at primes $p$. Let $C_0=C_0(\pi,k)>0$ be a sufficiently large constant depending only on $\pi$ and $k$.
	
	Let $\log_j$ denote the $j$-fold iterated logarithm and $\ell$ the largest
	integer such that $\log_{\ell}X\ge C_0$. Define an increasing sequence $X_j$ as follows.
	Put $X_1=(k+1)^{4}C_0$, and for $2\le j\le\ell$ put
	$$
	X_j=\exp\left(\frac{\log X}{C_0(k+1)^{2}(\log_j X)^{2}}\right).
	$$
	For $2\le j\le\ell$, set
	$$
	\mathcal{P}_{j,d}(s)=\sum_{X_{j-1}<p\le X_j}
	\frac{\lambda(p)\chi_{8d}(p)}{p^{s}}
	\qquad\text{and}\qquad
	P_j=\sum_{X_{j-1}<p\le X_j}\frac{\abs{\lambda(p)}^{2}}{p} .
	$$
	For each $2\le j\le\ell$, put
	$K_j=C_0(k+1)^{2}P_j$. Let
	$$
	\mathcal{S}_1 ~=~ \{n: p\vert n\implies p \le X_1\},\quad
	\mathcal{S}_\infty ~=~ \{n: p\vert n\implies p > X_{\ell}\},
	$$
	$$
	\mathcal{S}_j ~=~ \{n: p\vert n\implies X_{j-1} < p \le X_j\}
	$$
	and
	$$
	\mathcal{N}_j
	~=~
	\{n\in\mathcal{S}_j: \Omega(n)\le K_j\} 
	$$
	Let
	$$
	\mathcal{N}
	~=~
	\{n=n_2n_3\cdots n_{\ell}\in\N:\ n_j\in\mathcal{N}_j\};
	$$
	then for $n\in\mathcal{N}$ we have
	$$
	n=n_2\cdots n_{\ell}\le X_2^{K_2}X_3^{K_3}\cdots X_{\ell}^{K_{\ell}}
	\le X^{1/20} .
	$$
	Let $g$ denote the multiplicative function given on prime powers by
	$\displaystyle{g(p^{m})=\frac{\lambda(p)^{m}}{m!}}$. For any real number
	$\alpha$ and $2\le j\le\ell$ define
	$$
	\mathcal{N}_{j,d}(s,\alpha)
	~=~
	\sum_{0\le m\le K_j}\frac{1}{m!}\bigl(\alpha\mathcal{P}_{j,d}(s)\bigr)^{m}
	~=~
	\sum_{n\in\mathcal{N}_j}
	\frac{\alpha^{\Omega(n)}\chi_{8d}(n)g(n)}{n^{s}},
	$$
	and put
	$$
	\mathcal{N}_d(s,\alpha)
	=\prod_{j=2}^{\ell}\mathcal{N}_{j,d}(s,\alpha)
	=\sum_{n\in\mathcal{N}}\frac{\alpha^{\Omega(n)}\chi_{8d}(n)g(n)}{n^{s}} .
	$$
	
	Let $r=\floor{C_0}$ and set $\delta = \tfrac{1}{100r}$~. We choose $\phi = \phi_\delta$~, so that
	$$
	W_X(t)=(\log X)^{1/2}\,\widehat{\phi_\delta}(t\log X).
	$$
	For the sake of concision, we write
	$$
	\displaystyle \sumflat_{d\le X} = \sum_{\substack{C_0 \le d \le X \\ \mu^2(2d)=1}}~.
	$$ 
	The following quantity is the key ingredient in our lower bound mechanism,
	$$
	\mathcal{I}(X)=\sumflat_{d \le X} \int_{-\infty}^{\infty}\abs{L(\tfrac12+it,\pi \otimes \chi_{8d})}^{2/r}\abs{\mathcal{N}_d(\tfrac12+it,k-\tfrac1r)}^2\abs{W_X(t)}^2\, dt.
	$$
	We have by H\"older's inequality that 
	\begin{align*}
		\mathcal{I}(X)&\le\left( \sumflat_{d \le X}\int_{-\infty}^{\infty}\abs{L(\tfrac12+it,\pi \otimes \chi_{8d})}^{2k}\abs{W_X(t)}^2\,dt\right)^{\frac{1}{rk}} 
		\\
		&\qquad\times
		\left( \sumflat_{d \le X}\int_{-\infty}^{\infty}\abs{\mathcal{N}_d(\tfrac12+it,k-\tfrac1r)}^{\frac{2rk}{rk-1}}\abs{W_X(t)}^2\,dt \right)^{1-\frac{1}{rk}}.
	\end{align*}
	\Cref{thm:mainthm} then follows immediately from the two lemmas below.
	
	\begin{lem}\label{lem:Nlem}
		Let $X>0$ be large. We have
		$$
		\sumflat_{d \le X}\int_{-\infty}^{\infty}\abs{\mathcal{N}_d(\tfrac12+it,k-\tfrac1r)}^{\frac{2rk}{rk-1}}\abs{W_X(t)}^2\,dt \ll  X(\log X)^{2k^2}.
		$$
	\end{lem}
	
	\begin{lem}\label{lem:mainlem}
		Let $X>0$ be large. We have
		$$
		\mathcal{I}(X)\gg  X(\log X)^{2k^2+\frac{\epsilon}{r}}.
		$$
	\end{lem}
	
	We prove \Cref{lem:Nlem} in \Cref{sec:Nlem} by replacing $\abs{\mathcal{N}_d(\tfrac12+it,k-\tfrac1r)}^{\frac{2rk}{rk-1}}$ on average, with a manageable error term, by $\abs{\mathcal{N}_d(\tfrac12+it,k)}^{2}$. A mean value estimate then gives
	\begin{align*}
		\sumflat_{d \le X}\int_{-\infty}^{\infty}\abs{\mathcal{N}_d(\tfrac12+it,k)}^{2}\abs{W_X(t)}^2\,dt
		&\ll X~~\prod_{j=2}^\ell\sum_{\substack{m_jn_j = \square\\ m_j,n_j\in \mathcal{N}_j }}
		\frac{k^{\Omega(m_jn_j)}\abs{g(m_j)g(n_j)}}{\sqrt{m_jn_j}} 
		\\
		&\ll
		X~\prod_{X_1 < p \le X_\ell }\left( 1 + 2k^2 \frac{\lambda(p)^2}{p} \right)
		\\
		&\ll
		X(\log X)^{2k^2}.
	\end{align*}

	In \Cref{sec:mainlem} we prove \Cref{lem:mainlem} by considering the following integrals:
	$$
	\mathcal{J}(\sigma) = \sumflat_{d \le X}\int_{-\infty}^{\infty}\abs{L^*(\sigma+it,\pi \otimes \chi_{8d})}^{2/r}\abs{\mathcal{N}_d(\sigma+it,k-\tfrac1r)}^2 \abs{W_X(t-\sigma i +\tfrac12i)}^2\,dt,
	$$
	$$
	\mathcal{L}(\sigma) = \sumflat_{d \le X} \int_{-\infty}^{\infty}\abs{S_d(\sigma+it)}^{2} \abs{W_X(t-\sigma i +\tfrac12i)}^2\,dt,
	$$
	where $L^*(s,\pi\otimes\chi_{8d})$ denotes $L(s,\pi\otimes\chi_{8d})$ without its Euler factors at the primes $p\mid2q$, and $S_d(s)$ is a short Dirichlet polynomial, chosen so that 
	$$
	E_d(s) = L^*(s,\pi \otimes \chi_{8d})\mathcal{N}_d(s,k-\tfrac1r)^r-S_d(s)^r
	$$
	satisfies the pointwise bound $E_d(s) \ll X^{-\frac{1}{40}}$ on the line $\Re s=2$. For the quantity $\mathcal{L}(\sigma)$, we prove the following bounds:
	$$
	\mathcal{L}(1/2) \ll X(\log X)^{2k^2+\epsilon/r}
	~~\text{ and }~~
	\mathcal{L}(\sigma) \gg X(\sigma-1/2)^{-2k^2-\epsilon/r}
	$$
	uniformly for $\tfrac{1}{2} + \tfrac{c_0}{\log X} \le \sigma \le \tfrac{3}{4}$ and for some $c_0>0$ large enough. Finally, using a convexity theorem for integrals (applied to $L^*\mathcal{N}_d^r$ and to $E_d$), we deduce a bound of the form
	$$
	\mathcal{L}(\sigma) \ll X^{2(\sigma-1/2)}\mathcal{I}(X) + X^{-\delta(\sigma-1/2)}X(\log X)^{2k^2+\epsilon/r}.
	$$
	Choosing $\sigma = \tfrac{1}{2} + \tfrac{c}{\log X}$ for $c\ge c_0$, we get 
	$$
	X(\log X)^{2k^2+\epsilon/r} c^{-2k^2-\epsilon/r} \ll  e^{2c}\mathcal{I}(X)  + X(\log X)^{2k^2+\epsilon/r} e^{ -\delta c}.
	$$
	Hence, for $c>0$ sufficiently large, we see that 
	$$
	\mathcal{I}(X)\gg  X(\log X)^{2k^2+\frac{\epsilon}{r}}.
	$$

	\section{PRELIMINARIES}\label{sec:Preliminaries}
	In this section we collect the lemmas used in the proofs of \Cref{lem:Nlem} and \Cref{lem:mainlem}. Throughout,
	$$
	\rho(p^j)=\begin{cases}0,&p=2,\\[2pt]\frac{p}{p+1},&p>2,\end{cases}
	\quad (j\ge1),
	$$
	where $\rho$ is extended multiplicatively and $\rho(1)=1$. Thus, for odd $h$,
	$\rho(h)=\prod_{p\mid h}\tfrac{p}{p+1}$.
	
	\begin{lem}\label{lem:charlem}
		For odd $h\le X^2$,
		$$
		\sumflat_{d\le X}\chi_{8d}(h)=
		\begin{cases}
			\displaystyle \frac{4}{\pi^2}\rho(h)X+O_\varepsilon(X^{1/2+\varepsilon}),
			& h=\square,\\[5pt]
			O_\varepsilon(X^{1/2+\varepsilon}h^{1/4}),&h\ne\square.
		\end{cases}
		$$
	\end{lem}
	
	\begin{proof}
		Since $h$ is odd, we have
		\begin{align*}
			\sumflat_{d \le X} \chi_{8d}(h) 
			&= \left(\frac{8}{h}\right) \sum_{\substack{d \le X \\ (d,2)=1}} \mu^2(d) \left(\frac{d}{h}\right) +O(1)
			\\
			&= \left(\frac{8}{h}\right) \sum_{\substack{d \le X \\ (d,2)=1}} \left( \sum_{a^2\vert d}\mu(a) \right) \left(\frac{d}{h}\right)+O(1)
			\\
			&= \left(\frac{8}{h}\right) \sum_{\substack{a \le \sqrt{X} \\ (a,2)=1}} \mu(a) \left(\frac{a^2}{h}\right) \sum_{\substack{b \le X/a^2 \\ (b,2)=1}} \left(\frac{b}{h}\right) +O(1).
		\end{align*}
		If $h=\square$, then 
		$$
		\sum_{\substack{b \le X/a^2 \\ (b,2)=1}} \left(\frac{b}{h}\right) =\sum_{\substack{b \le X/a^2 \\ (b,2h)=1}} 1 = \frac{\phi(2h)}{2h} \frac{X}{a^2} + O_\varepsilon(X^{\varepsilon}).
		$$
		Hence,
		\begin{align*}
			\sumflat_{d \le X} \chi_{8d}(h) 
			&=X \frac{\phi(2h)}{2h} \sum_{\substack{a =1 \\ (a,2h)=1}}^\infty \frac{\mu(a)}{a^2} + O_\varepsilon(X^{1/2+\varepsilon})
			\\
			&= \frac{X}{\zeta(2)} \frac{\phi(2h)}{2h} \prod_{p\vert 2h} \left( 1 - \frac{1}{p^2} \right)^{-1} + O_\varepsilon(X^{1/2+\varepsilon})
			\\
			&= \frac{4}{\pi^2}\rho(h)X + O_\varepsilon(X^{1/2+\varepsilon}) .
		\end{align*}
		
		If $h\neq\square$, then $b\mapsto\left(\frac{b}{h}\right)$ is a non-principal character modulo $h$, and
		$$
		\sum_{\substack{b \le X/a^2 \\ (b,2)=1}} \left(\frac{b}{h}\right) = \sum_{b \le X/a^2} \left(\frac{b}{h}\right) - \left(\frac{2}{h}\right) \sum_{b \le X/(2a^2)} \left(\frac{b}{h}\right).
		$$
		By the Pólya--Vinogradov inequality and the trivial bound, we see that 
		$$
		\sum_{\substack{b \le X/a^2 \\ (b,2)=1}} \left(\frac{b}{h}\right)
		\ll \min\left( \frac{X}{a^2}, h^{1/2} \log h \right).
		$$	 
		We split the outer sum over $a$ at a parameter $A \le \sqrt{X}$, yielding
		\begin{align*}
			\abs{ \sumflat_{d \le X} \chi_{8d}(h) } &\ll \sum_{a \le A} h^{1/2} \log h + \sum_{A < a \le \sqrt{X}} \frac{X}{a^2} 
			\\
			&\ll A h^{1/2} \log h + \frac{X}{A}.
		\end{align*}
		Choosing $A = X^{1/2} h^{-1/4}$ gives
		$$
		\abs{ \sumflat_{d \le X} \chi_{8d}(h) } \ll X^{1/2} h^{1/4} \log h \ll_\varepsilon X^{1/2+\varepsilon} h^{1/4},
		$$
		since $h\le X^2$. This completes the proof of \Cref{lem:charlem}.
	\end{proof}
	
	\begin{lem}\label{lem:mvt}
		For any complex coefficients $a(m,n)$ with $m,n\le X$,
		\begin{align*}
			\left|\sumflat_{d\le X}\int_{-\infty}^{\infty}
			\sum_{m,n\le X}\chi_{8d}(mn)a(m,n)(n/m)^{it}|W_X(t)|^2\,dt\right|
			&\ll_\varepsilon X\sum_{\substack{m,n\le X\\mn=\square}}|a(m,n)|\\
			&\quad+X^{1/2+\varepsilon}
			\sum_{m,n\le X}|a(m,n)|(mn)^{1/4}.
		\end{align*}
	\end{lem}
	
	\begin{proof}
		
		Since $\int_{-\infty}^{\infty}\abs{W_X(t)}^2dt \ll 1$, we see that 
		\begin{align*}
			\abs{\sumflat_{d\le X}\int_{-\infty}^{\infty} \sum_{m,n\le X} \chi_{8d}(mn)a(m,n)(n/m)^{it} \abs{ W_X(t) }^2 dt}\ll 
			\sum_{m,n\le X} \abs{ a(m,n) } \abs{\sumflat_{d \le X} \chi_{8d}(mn)}.
		\end{align*}
		\Cref{lem:mvt} then follows from \Cref{lem:charlem}.
		
	\end{proof}
	
	For sequences $a_1$ and $a_2$ of complex numbers we define
	$$
	\langle a_1,a_2\rangle
	:=\lim_{Y\to\infty}
	\sum_{\substack{l_1,l_2\le Y\\l_1l_2=\square}}
	a_1(l_1)\overline{a_2(l_2)}\rho(l_1l_2),
	$$
	whenever this limit exists and is finite, and set $\|a\|=\langle a,a\rangle^{1/2}$.
	For fixed $Y$, \Cref{lem:charlem} gives
	\begin{equation}\label{eq:inner-product-mean}
		\langle a_1\mathds{1}_{l\le Y},a_2\mathds{1}_{l\le Y}\rangle
		=\lim_{X\to\infty}\frac{\pi^2}{4X}
		\sumflat_{d\le X}
		\left(\sum_{l\le Y}a_1(l)\chi_{8d}(l)\right)
		\overline{\left(\sum_{l\le Y}a_2(l)\chi_{8d}(l)\right)}.
	\end{equation}
	In particular, $\|a\mathds{1}_{l\le Y}\|^2\ge0$ for every sequence $a$ and every $Y$.
	
	\begin{lem}\label{lem:inequalitylem}
		If the pairings and norms below exist and are finite, then
		\begin{align}
			|\langle a_1,a_2\rangle|&\le\|a_1\|\,\|a_2\|,\label{eq:cauchy}\\
			\|a_1+a_2\|&\le\|a_1\|+\|a_2\|.\label{eq:minkowski}
		\end{align}
	\end{lem}
	
	\begin{proof}
		Fix $Y$ and use \eqref{eq:inner-product-mean}. Cauchy--Schwarz for the finite
		sum over $d$ gives
		\[
		|\langle a_1\mathds{1}_{l\le Y},a_2\mathds{1}_{l\le Y}\rangle|
		\le\|a_1\mathds{1}_{l\le Y}\|\,\|a_2\mathds{1}_{l\le Y}\|.
		\]
		Letting $Y\to\infty$ proves \eqref{eq:cauchy}. For each $Y$, we have
		\begin{align*}
			\norm{a_1\mathds{1}_{l\le Y}+a_2\mathds{1}_{l\le Y}}^2 &= \langle a_1\mathds{1}_{l\le Y}, a_1\mathds{1}_{l\le Y}+a_2\mathds{1}_{l\le Y} \rangle + \langle a_2\mathds{1}_{l\le Y}, a_1\mathds{1}_{l\le Y}+a_2\mathds{1}_{l\le Y} \rangle.
		\end{align*}
		Applying \eqref{eq:cauchy} and letting $Y\to\infty$ proves \eqref{eq:minkowski}.
	\end{proof}
	
	\begin{lem}\label{lem:multilem}
		Let $a_1,a_2$ be multiplicative functions and suppose
		\[
		\sum_{l_1l_2=\square}
		|a_1(l_1)a_2(l_2)|\rho(l_1l_2)<\infty.
		\]
		Then 
		\[
		\langle a_1,a_2\rangle
		=\prod_p\left(1+\rho(p)
		\sum_{\substack{a,b\ge0\\a+b>0\text{ even}}}
		a_1(p^a)\overline{a_2(p^b)}\right).
		\]
		More generally, if the primes are split into finitely many disjoint sets and
		$a_\nu(l)=\prod_i a_{\nu,i}(l^{(i)})$ for $\nu =1,2$, where $l^{(i)}$ is the part of $l$ supported on the
		$i$-th set, then $\langle a_1,a_2\rangle=\prod_i\langle a_{1,i},a_{2,i}\rangle$, provided each
		pairing on the right converges absolutely.
	\end{lem}
	
	\begin{proof}
		Immediate from the multiplicativity of $a_1$, $a_2$ and $\rho$.
	\end{proof}
	
	For $\beta>0$, define the multiplicative function $\tau_\beta$ by
	$$
	\tau_\beta(1)=1,\qquad
	\tau_\beta(p^m)=\frac{\Gamma(m+\beta)}{\Gamma(\beta)m!}\quad(m\ge1).
	$$
	Thus
	$(1-z)^{-\beta}=\sum_{m\ge0}\tau_\beta(p^m)z^m$ as a formal power series.
	
	\begin{lem}\label{lem:proplem1}
		Let $\beta>0$. For $p\nmid q$ and $\mu\ge0$, put 
		$$
		a(p^\mu)=\sum_{j=1}^D\alpha_j(p)^\mu.
		$$
		Define the
		multiplicative coefficients $\lambda_\beta(n)$ formally by
		$\prod_{p\nmid q}L_p(s,\pi)^\beta=\sum_{n\ge1}\lambda_\beta(n)n^{-s}$, so that $\lambda_\beta(n)=0$
		whenever $(n,q)>1$. Then, for $p\nmid q$,
		\[
		\lambda_\beta(p^m)=
		\sum_{\substack{m_1+\cdots+m_D=m\\m_j\ge0}}
		\prod_{j=1}^D\alpha_j(p)^{m_j}\tau_\beta(p^{m_j}),
		\]
		and
		\[
		\lambda_\beta(p^m)=
		\sum_{\substack{v_\mu\ge0,\,\text{finitely supported}\\
				\sum_{\mu\ge1}\mu v_\mu=m}}
		\prod_{\mu\ge1}
		\frac{(\beta a(p^\mu))^{v_\mu}}{\mu^{v_\mu}v_\mu!}.
		\]
		In particular, $\lambda_\beta(p)=\beta\lambda(p)$ and $\lambda_\beta(p^2)=\tfrac12\big(\beta^2\lambda(p)^2+\beta a(p^2)\big)$.
	\end{lem}
	
	\begin{proof}
		At each prime $p\nmid q$, the local generating function is
		\[
		\sum_{m\ge0}\lambda_\beta(p^m)x^m
		=\prod_{j=1}^D(1-\alpha_j(p)x)^{-\beta}.
		\]
		Expanding each factor by the binomial series and comparing coefficients gives
		the first formula. Taking the formal logarithm instead gives
		\[
		\prod_{j=1}^D(1-\alpha_j(p)x)^{-\beta}
		=\exp\left(\beta\sum_{\mu\ge1}\frac{a(p^\mu)}{\mu}x^\mu\right)
		=\prod_{\mu\ge1}\sum_{v_\mu\ge0}
		\frac1{v_\mu!}
		\left(\frac{\beta a(p^\mu)}{\mu}x^\mu\right)^{v_\mu}.
		\]
		The coefficient of $x^m$ is obtained from the choices with
		$\sum_\mu\mu v_\mu=m$, giving the second formula. 
	\end{proof}
	
	\begin{lem}\label{lem:proplem2}
		Put $A(p)=\max_j|\alpha_j(p)|$ for $p\nmid q$, and $A(p)=0$ for $p\mid q$. There is
		$\theta_0=\theta_0(\pi)\in(0,1/2)$ such that $A(p)\le p^{\theta_0}$ for all
		primes $p$. For fixed $\beta>0$ and every $m\ge0$,
		\[
		|\lambda_\beta(p^m)|\le A(p)^m\tau_{D\beta}(p^m),\qquad
		|\lambda_\beta(n)|\le n^{\theta_0}\tau_{D\beta}(n).
		\]
		Consequently, for every $\theta\in(\theta_0,1/2)$, we have
		$\lambda_\beta(n)\ll_{\pi,\beta,\theta}n^\theta$. Also, for every $\sigma>1$,
		\[
		\sum_p\frac{A(p)^2}{p^\sigma}<\infty.
		\]
	\end{lem}
	
	\begin{proof}
		For $p\nmid q$, the bound for $A(p)$ is the Luo--Rudnick--Sarnak estimate \cite{LRS}. Taking
		absolute values in the first formula of \Cref{lem:proplem1} and using
		$|\alpha_j(p)|\le A(p)$ gives
		$$
		|\lambda_\beta(p^m)|\le A(p)^m
		\sum_{m_1+\cdots+m_D=m}\prod_j\tau_\beta(p^{m_j})
		=A(p)^m\tau_{D\beta}(p^m).
		$$
		The final prime sum estimate follows from a result of Brumley \cite[Cor.~8]{Brumley}.
	\end{proof}

	\begin{lem}\label{lem:sqlem}
		For an unramified prime $p$, put
		$$
		B(p)=\lambda_{\wedge^2\pi}(p)=\sum_{i<j}\alpha_i(p)\alpha_j(p).
		$$ 
		Then
		$$
		\lambda(p^2)=\lambda(p)^2-B(p),\qquad
		a(p^2)=\lambda(p)^2-2B(p).
		$$
		Moreover, if $D\in\{2,3\}$, then 
		$$
		|\lambda(p^2)|\le2(1+|\lambda(p)|^2),\qquad
		|a(p^2)|\le3(1+|\lambda(p)|^2).
		$$
	\end{lem}
	
	\begin{proof}
		Write $\alpha_j=\alpha_j(p)$. The local Euler factor gives
		$$
		\lambda(p)=\sum_j\alpha_j,~\lambda(p^2)=\sum_{i\le j}\alpha_i\alpha_j
		\text{ and }a(p^2)=\sum_j\alpha_j^2.
		$$ 
		Expanding $\lambda(p)^2$ proves both
		identities.
		
		For $D=2$, $B(p)=\omega_\pi(p)$ and $|B(p)|=1$. For $D=3$, we have (see \cite[Introduction]{Kim})
		$$
		\wedge^2\pi\cong\widetilde\pi\otimes\omega_\pi,
		$$
		so that $|B(p)|=|\lambda(p)|$. Hence, the bounds for $|\lambda(p^2)|$ and $|a(p^2)|$ follow.
	\end{proof}
	
	For $D=4$, the pointwise bounds in \Cref{lem:sqlem} need not hold. We use
	the following averaged substitute.
	\begin{lem}\label{lem:avsqlem}
		Assume $D=4$. There is $c_\pi>0$ such that, for $x\ge3$,
		\[
		\sum_{p\le x,\,p\nmid q}\frac{|B(p)|^2}{p}
		=c_\pi\log\log x+O(1).
		\]
		Consequently, for $3\le Y<Z$ and $0<u\le\tfrac12$,
		\[
		\sum_{Y<p\le Z,\,p\nmid q}\frac{|B(p)|}{p}
		\ll 1+\log\log Z-\log\log Y,
		\qquad
		\sum_{p\nmid q}\frac{|B(p)|}{p^{1+u}}
		\ll\log\frac1u.
		\]
	\end{lem}
	
	\begin{proof}
		
		There is an automorphic representation $\Pi$ of $\mathrm{GL}_6(\mathbb{A}_\Q)$ such that $\lambda_\Pi(p)=B(p)$ for all primes $p\nmid q$ (see \cite[Theorem A]{Kim} and \cite{WuYe}). By Rankin--Selberg theory \cite{JS81,JPSS}, the $L$-function $L(s,\Pi\times\widetilde\Pi)$ has a pole at $s=1$ of some positive integer order $c_\pi$. Its Rankin--Selberg Mertens estimate (see \cite[Corollary 1.5]{LWY} and \cite{WuYe}) gives
		$$
		\sum_{p\le x,\,p\nmid q}
		\frac{|\lambda_\Pi(p)|^2}{p}=c_\pi\log\log x+O(1).
		$$
		The last two estimates follow immediately from Cauchy--Schwarz and partial summation respectively.
	\end{proof}
	
	\begin{lem}\label{lem:proplem3}
		For $0<|w|\le 1$ with $\Re w>0$,
		\[
		\Re\sum_{p\nmid q}\frac{\lambda(p)^2}{p^{1+w}}
		=\log\frac1{|w|}+O(1),\qquad
		\Re\sum_{p\nmid q}\frac{a(p^2)}{p^{1+w}}
		=\epsilon\log\frac1{|w|}+O(1),
		\]
		and, for $x\ge3$,
		\[
		\sum_{p\le x}\frac{\lambda(p)^2}{p}=\log\log x+O(1).
		\]
	\end{lem}
	
	\begin{proof}
		Throughout the proof, $L$-functions are Euler products over the primes $p\nmid q$ only.
		
		Let $L(s)$ be either $L(s,\pi\times\widetilde\pi)$ or $L(s,\wedge^2\pi)$. For $p\nmid q$ its
		Euler factor is 
		$$
		L_p(s)=\prod_\beta\left(1-\frac{\beta}{p^{s}}\right)^{-1},
		$$ where $\beta$ runs over the products
		$\alpha_i(p)\alpha_j(p)$ with $1\le i,j\le D$, respectively
		with $1\le i<j\le D$. Thus $\sum_\beta\beta=c(p)$, where $c(p)=\lambda(p)^2$, respectively
		$c(p)=B(p)$. By \Cref{lem:proplem2},
		for $\Re s\ge1$ we have
		\[
		\Big|\log L_p(s)-\frac{c(p)}{p^{s}}\Big|
		\le\sum_\beta\sum_{m\ge2}\frac{|\beta|^m}{p^{m}}
		\le D^2\sum_{m\ge2}\Big(\frac{A(p)^2}{p}\Big)^m
		\ll\frac{A(p)^4}{p^2}\le\frac{A(p)^2}{p^{2-2\theta_0}}.
		\]
		Since $|c(p)|\le D^2A(p)^2$ and $\sum_pA(p)^2p^{-\sigma}<\infty$ for every $\sigma>1$, the Euler
		product for $L(s)$ converges absolutely for $\Re s>1$. Hence $L(s)$ is holomorphic and non-zero there, and
		\begin{equation}\label{eq:logL}
			\log L(s)=\sum_{p\nmid q}\frac{c(p)}{p^{s}}+O(1)\qquad(\Re s>1).
		\end{equation}
		
		Next we consider the line $\Re s=1$. The function $L(s,\pi\times\widetilde\pi)$ continues
		meromorphically, is holomorphic on $\Re s\ge1$ except for a simple pole at $s=1$
		\cite{JS81,JPSS}, and does not vanish on $\Re s=1$ \cite{Shahidi81}. We claim that
		$L(s,\wedge^2\pi)$ is meromorphic, does not vanish on $\Re s=1$, and is holomorphic on $\Re s=1$
		except possibly for a simple pole at $s=1$. For $D=2$ we have $L(s,\wedge^2\pi)=L(s,\omega_\pi)$,
		where $\omega_\pi^2=1$ because $\pi\cong\widetilde\pi$, and the claim is classical. For $D=3$,
		$L(s,\wedge^2\pi)=L(s,\widetilde\pi\otimes\omega_\pi)$ is entire \cite{GJ}. For
		$D=4$, $L(s,\wedge^2\pi)$ is the $L$-function of Kim's exterior square lift
		\cite[Theorem~A]{Kim}, hence meromorphic. Since $\pi\cong\widetilde\pi$,
		\[
		L(s,\pi\times\widetilde\pi)=L(s,\pi,\mathrm{Sym}^2)\,L(s,\wedge^2\pi),
		\]
		so $L(s,\pi,\mathrm{Sym}^2)$ is meromorphic as well, and neither factor vanishes on $\Re s=1$
		\cite{Shahidi81}. As $L(s,\pi,\mathrm{Sym}^2)$ has no zeros on $\Re s=1$, this proves the claim. Let $\kappa\in\{0,1\}$ be the order of the pole of
		$L(s,\wedge^2\pi)$ at $s=1$. By the definition of $\epsilon$, we have $\epsilon=1-2\kappa$.
		
		Now let $k$ be the order of the pole of $L(s)$ at $s=1$, so that $k=1$ for
		$L(s,\pi\times\widetilde\pi)$ and $k=\kappa$ for $L(s,\wedge^2\pi)$. Taking real parts in
		\eqref{eq:logL} at $s=1+w$ gives, for $0<|w|\le1$ with $\Re w>0$,
		\[
		\Re\sum_{p\nmid q}\frac{c(p)}{p^{1+w}}
		=\log|L(1+w)|+O(1)
		=k\log\frac1{|w|}+O(1).
		\]
		For $L(s,\pi\times\widetilde\pi)$ this is 
		\[
		\Re\sum_{p\nmid q}\frac{\lambda(p)^2}{p^{1+w}}
		=\log\frac1{|w|}+O(1).
		\]
		For $L(s,\wedge^2\pi)$ it gives
		$$
		\Re\sum_{p\nmid q}\frac{B(p)}{p^{1+w}}=\kappa\log\frac{1}{|w|}+O(1),
		$$ 
		from $a(p^2)=\lambda(p)^2-2B(p)$ and $\epsilon=1-2\kappa$ it follows that
		\[
		\Re\sum_{p\nmid q}\frac{a(p^2)}{p^{1+w}}
		=\epsilon\log\frac1{|w|}+O(1).
		\]
		Finally, Selberg orthogonality
		\cite[Corollary 1.5]{LWY} gives
		\[
		\sum_{p\le x}\frac{\lambda(p)^2}{p}=\log\log x+O(1).
		\qedhere
		\]
	\end{proof}

	\begin{lem}\label{lem:proplem4}
		Let $L_p(s,\pi\otimes\chi_{8d})$ be the local factor at $p$ and put
		\[
		L^*(s,\pi\otimes\chi_{8d})
		:=\prod_{(p,2q)=1}L_p(s,\pi\otimes\chi_{8d}).
		\]
		For $\Re s>1$,
		\[
		L^*(s,\pi\otimes\chi_{8d})
		=\sum_{(l,2q)=1}\frac{\lambda(l)\chi_{8d}(l)}{l^s}.
		\]
		Moreover, $L^*(s,\pi\otimes\chi_{8d})$
		continues to an entire function, and
		\[
		|L(s,\pi\otimes\chi_{8d})|\gg_\pi
		|L^*(s,\pi\otimes\chi_{8d})|
		\]
		uniformly in $d$ and in $\Re s\ge\tfrac12$.
	\end{lem}
	
	\begin{proof}
		Recall that $d$ is odd and squarefree, so $\chi_{8d}$ is ramified exactly at $2$ and at the
		primes dividing $d$. Let $p\nmid2q$ and $\Re s>1$. If
		$p\nmid d$, then $\pi$ and $\chi_{8d}$ are both unramified at $p$, so
		$$
		L_p(s,\pi\otimes\chi_{8d})=\prod_{j=1}^D\left(1-\frac{\alpha_j(p)\chi_{8d}(p)}{p^s}\right)^{-1}.
		$$ If $p\mid d$ then $\chi_{8d}(p)=0$ and
		$L_p(s,\pi\otimes\chi_{8d})=1$, so the
		same formula holds. 
		
				For $p\mid2q$, since $\pi\otimes\chi_{8d}$ is unitary and cuspidal,
		$L_p(s,\pi\otimes\chi_{8d})^{-1}=\prod_{j=1}^D(1-\gamma_jp^{-s})$ with $|\gamma_j|<p^{1/2}$
		(see \cite[Proposition 3.3]{MS}). Thus
		$$
		L^*(s,\pi\otimes\chi_{8d})=L(s,\pi\otimes\chi_{8d})\prod_{p\mid2q}L_p(s,\pi\otimes\chi_{8d})^{-1}
		$$
		is entire \cite{GJ}, and for $\Re s\ge0$,
		$$
		|L^*(s,\pi\otimes\chi_{8d})|\le\prod_{p\mid2q}(1+p^{1/2})^D\,|L(s,\pi\otimes\chi_{8d})|.
		$$
	\end{proof}

	\section{PROOF OF LEMMA \ref{lem:Nlem}}
	
	\label{sec:Nlem}

	To prove \Cref{lem:Nlem} we follow the arguments of Heap and Soundararajan \cite[Proposition 3]{HeapSoundararajan2022}. We need the following lemmas.
	
	\begin{lem}\label{lem:Npowerlem}
		For $2\le j\le\ell$ and every $d$,
		$$
		\abs{\mathcal{N}_{j,d}(\tfrac12+it,k-\tfrac1r)}^{\frac{2rk}{rk-1}}
		\le
		\abs{\mathcal{N}_{j,d}(\tfrac12+it,k)}^{2}\bigl(1+7e^{-K_j}\bigr)
		+\mathcal{Q}_{j,d}(t)
		$$
		where
		$$
		\mathcal{Q}_{j,d}(t)=\left(\frac{12(k+1)\abs{\mathcal{P}_{j,d}(\tfrac12+it)}}{K_j}\right)^{4K_j}.
		$$
	\end{lem}
	
	\begin{proof}
		
		For $\abs{z}\le K/10$,
		\begin{equation}\label{approximation}
			\abs{\sum_{m=0}^{K}\frac{z^{m}}{m!}-e^{z}}\le\frac{\abs{z}^{K}}{K!}
			\le\left(\frac{e}{10}\right)^{K}.
		\end{equation}
		In the first case, where $\abs{\mathcal{P}_{j,d}(\tfrac12+it)}\le K_j/(10k+10)$, we
		have by \eqref{approximation} that
		\begin{align*}
			\abs{\mathcal{N}_{j,d}(\tfrac12+it,k-\tfrac1r)}^{\frac{2rk}{rk-1}}
			&=\abs{\sum_{0\le m\le K_j}\frac{1}{m!}
				(k-\tfrac1r)^m\bigl(\mathcal{P}_{j,d}(\tfrac12+it)\bigr)^{m}}^{\frac{2rk}{rk-1}}
			\\
			&\le\abs{\exp\bigl(2k\mathcal{P}_{j,d}(\tfrac12+it)\bigr)}
			\bigl(1+e^{-K_j}\bigr)^{\frac{2rk}{rk-1}}
			\\
			&\le\abs{\mathcal{N}_{j,d}(\tfrac12+it,k)}^{2}\bigl(1+7e^{-K_j}\bigr).
		\end{align*}
		The lemma follows in that case. In the second case, where
		$\abs{\mathcal{P}_{j,d}(\tfrac12+it)}\ge K_j/(10k+10)$,
		\begin{align*}
			\abs{\mathcal{N}_{j,d}(\tfrac12+it,k-\tfrac1r)}
			&\le\sum_{0\le m\le K_j}\frac{1}{m!}
			(k-\tfrac1r)^m\bigl(\abs{\mathcal{P}_{j,d}(\tfrac12+it)}\bigr)^{m}\\
			&\le\abs{(k+1)\mathcal{P}_{j,d}(\tfrac12+it)}^{K_j}
			\sum_{0\le m\le K_j}\frac{1}{m!}(10/K_j)^{K_j-m}\\
			&\le\left(\frac{12(k+1)\abs{\mathcal{P}_{j,d}(\tfrac12+it)}}{K_j}\right)^{K_j},
		\end{align*}
		and raising to the power $\frac{2rk}{rk-1} \le 4$ gives $\mathcal{Q}_{j,d}(t)$.
	\end{proof}

	\begin{lem}\label{lem:Qlem}
		
		With the above notation and $X>0$ large, we have for $2\le j\le\ell$
		$$
		\sumflat_{d \le X}\int_{-\infty}^{\infty}\mathcal{Q}_{j,d}(t)
		\abs{W_X(t)}^{2}\,dt~\ll~ X e^{-K_j}.
		$$
	\end{lem}
	
	\begin{proof}
		
		We have
		$$
		\mathcal{P}_{j,d}(s)^{2K_j}
		=\sum_{\substack{p\mid n\implies X_{j-1}<p\le X_j\\ \Omega(n)=2K_j}}
		\frac{(2K_j)!\,\chi_{8d}(n)g(n)}{n^{s}},
		$$
		a Dirichlet polynomial of length $X_j^{2K_j}\le X^{1/10}$. 
		We have by \Cref{lem:mvt},
		\begin{align*}
			\sumflat_{d \le X}\int_{-\infty}^{\infty}\mathcal{Q}_{j,d}(t)
			\abs{W_X(t)}^{2}\,dt
			&\ll 
			\left(\frac{12(k+1)}{K_j}\right)^{4K_j}\bigl((2K_j)!\bigr)^{2}
			X \sum_{\substack{mn=\square\\ \Omega(m)=\Omega(n)=2K_j}}
			\frac{\abs{g(m)g(n)}}{\sqrt{mn}} .
		\end{align*}
		We see that
		\begin{align*}
			\sum_{\substack{mn=\square\\ \Omega(m)=\Omega(n)=2K_j}}
			\frac{\abs{g(m)g(n)}}{\sqrt{mn}} 
			&\le
			\sum_{\sum_p c_p= 2K_j}  \prod_p \left( \frac{\lambda(p)^{2c_p}}{p^{c_p}}\sum_{a+b=2c_p}\frac{1}{a!b!} \right)
			\\
			&= \sum_{\sum_p c_p= 2K_j}  \prod_p \frac{2^{2c_p}\lambda(p)^{2c_p}}{(2c_p)!p^{c_p}}
			\\
			&\le \sum_{\sum_p c_p= 2K_j}  \prod_p \frac{1}{c_p!}\left(\frac{2\lambda(p)^{2}}{p}\right)^{c_p}
			\\
			&= \frac{(2P_j)^{2K_j}}{(2K_j)!}.
		\end{align*}
		Using above and the fact that $P_j=K_j/(C_0(k+1)^{2})$, Stirling's formula gives
		\begin{align*}
			\sumflat_{d \le X}\int_{-\infty}^{\infty}\mathcal{Q}_{j,d}(t)
			\abs{W_X(t)}^{2}\,dt
			&\ll
			X \left(\frac{12(k+1)}{K_j}\right)^{4K_j} (2K_j)! (2P_j)^{2K_j} 
			\\
			&\ll
			X e^{-K_j}.
		\end{align*}
	\end{proof}
	
	We now deduce \Cref{lem:Nlem} from \Cref{lem:Npowerlem} and \Cref{lem:Qlem} as follows.
	
	\begin{proof}[Proof of \Cref{lem:Nlem}]
		By \Cref{lem:Npowerlem}, expanding the product over $2\le j\le\ell$,
		$$
		\abs{\mathcal{N}_{d}(\tfrac12+it,k-\tfrac1r)}^{\frac{2rk}{rk-1}}
		\le\prod_{j=2}^{\ell}\left(
		\abs{\mathcal{N}_{j,d}(\tfrac12+it,k)}^{2}\bigl(1+7e^{-K_j}\bigr)
		+\mathcal{Q}_{j,d}(t)\right).
		$$
		Let 
		$$
		a^{\mathcal{N}}_j(m,n) = \frac{(1+7e^{-K_j} )k^{\Omega(mn)}g(m)\overline{g(n)} }{ (mn)^{1/2} }\mathds{1}_{m,n\in\mathcal{N}_j}
		$$
		and 
		$$
		a^{\mathcal{Q}}_j(m,n) = ((2K_j)!)^2 \left(\frac{12(k+1)}{K_j}\right)^{4K_j} \frac{ g(m)\overline{g(n)} }{ (mn)^{1/2} }\mathds{1}_{\substack{m,n \in \mathcal{S}_j \\ \Omega(m)=2K_j \\ \Omega(n)=2K_j}}~,
		$$
		then
		$$
		\abs{\mathcal{N}_{j,d}(\tfrac12+it,k)}^{2} \bigl(1+7e^{-K_j}\bigr)= \sum_{ m,n } \chi_{8d}(mn) a^{\mathcal{N}}_j(m,n) (n/m)^{it}
		$$
		and 
		$$
		\mathcal{Q}_{j,d}(t)= \sum_{ m,n } \chi_{8d}(mn) a^{\mathcal{Q}}_j(m,n) (n/m)^{it}.
		$$
		Put $a_j(m,n)= a^{\mathcal{N}}_j(m,n) + a^{\mathcal{Q}}_j(m,n)$ Then by \Cref{lem:mvt} we have, 
		\begin{align*}
			\sumflat_{d \le X}\int_{-\infty}^{\infty} & \prod_{j=2}^{\ell}\left(
			\abs{\mathcal{N}_{j,d}(\tfrac12+it,k)}^{2}\bigl(1+7e^{-K_j}\bigr)
			+\mathcal{Q}_{j,d}(t)\right) \abs{W_X(t)}^{2}\,dt 
			\\
			&\ll
			X\sum_{ mn=\square} \abs{a(m,n)},
		\end{align*}
		where 
		$$
		a(m,n)=\prod_{j=2}^{\ell} a_j(m_j,n_j)
		$$
		and $m_j,n_j\in\mathcal{S}_j$ satisfy $\prod m_j = m$ and $\prod n_j = n$. We have $mn=\square$ iff $m_j n_j =\square$,
		$$
		\sum_{mn=\square} \abs{a(m,n)} \le \prod_{j=2}^{\ell}\sum_{\substack{m_j,n_j \in \mathcal{S}_j \\ m_j n_j =\square}} \left(|a_j^{\mathcal{N}}(m_j,n_j)| \quad+|a_j^{\mathcal{Q}}(m_j,n_j)| \right).
		$$
		From \Cref{lem:proplem2}, we see that 
		\begin{align*}
			\sum_{\substack{m_j,n_j \in \mathcal{S}_j \\ m_j n_j =\square}} \abs{a_j^{\mathcal{N}}(m_j,n_j)} &\le (1+7e^{-K_j}) \prod_{X_{j-1} < p \le X_j} \sum_{\substack{ a,b \ge 0 \\ a+b\text{ even}}} \frac{k^{a+b}\abs{g(p^a)g(p^b)}}{p^{(a+b)/2}} 
			\\
			&=
			(1+7e^{-K_j}) \prod_{X_{j-1} < p \le X_j} \left( 1 + 2k^2 \frac{\lambda(p)^2}{p}  + O\left( \frac{ A(p)^2 }{p^{2(1-\theta)}} \right) \right).
		\end{align*}
		As in the proof of \Cref{lem:Qlem},
		\begin{align*}
			\sum_{\substack{m_j,n_j \in \mathcal{S}_j \\ m_j n_j =\square}} \abs{a_j^{\mathcal{Q}}(m_j,n_j)} &\le ((2K_j)!)^2 \left(\frac{12(k+1)}{K_j}\right)^{4K_j}   \sum_{\substack{m_j,n_j \in \mathcal{S}_j \\ \Omega(m_j) = \Omega(n_j) = 2K_j  \\ m_j n_j =\square}}
			\frac{\abs{g(m_j)g(n_j)}}{\sqrt{m_jn_j}} 
			\\
			&\ll e^{-K_j}.
		\end{align*}
		Hence from \Cref{lem:proplem3}, 
		\begin{align*}
			\sum_{mn=\square} \abs{a(m,n)} 
			&\le
			\prod_{X_1 < p \le X_\ell }\left( 1 + 2k^2 \frac{\lambda(p)^2}{p}  + O\left( \frac{ A(p)^2 }{p^{2(1-\theta)}} \right) \right)
			\\
			&\ll
			(\log X)^{2k^2}.
		\end{align*}
	    This finishes the proof of \Cref{lem:Nlem}.
		
	\end{proof}

	\section{PROOF OF LEMMA \ref{lem:mainlem}}
	\label{sec:mainlem}
	
	By \Cref{lem:proplem4} we have $|L(\tfrac12+it,\pi\otimes\chi_{8d})|\gg|L^*(\tfrac12+it,\pi\otimes\chi_{8d})|$ uniformly in $d$ and $t$. Thus, to prove \Cref{lem:mainlem} it is sufficient to prove the following.
	\begin{lem}\label{lem:mainlem*}
		Let $X>0$ be large. We have
		$$
		\mathcal{I}^*(X) \gg X(\log X)^{2k^2+\epsilon/r},
		$$ 
		where 
		$$
		\mathcal{I}^*(X)=\sumflat_{d \le X} \int_{-\infty}^{\infty}\abs{L^*(\tfrac12+it,\pi \otimes \chi_{8d})}^{2/r}\abs{\mathcal{N}_d(\tfrac12+it,k-\tfrac1r)}^2\abs{W_X(t)}^2\, dt.    
		$$
	\end{lem}
	To prove \Cref{lem:mainlem*} we utilize the technique of Heath-Brown \cite{HeathBrown1981,HeathBrown2010}. Let
	$$
	b(l) = \sum_{\substack{mn=l \\ n\in\mathcal{N}}}\lambda_{1/r}(m) (k-1/r)^{\Omega(n)}g(n)
	$$
	and
	$$
	b'(l) = \sum_{\substack{mn=l \\ p\vert n \implies X_1 < p \le X_\ell}}\lambda_{1/r}(m) (k-\tfrac1r)^{\Omega(n)}g(n).
	$$
	Let $v : \R \to [0,1]$ be a smooth function such that $v (x) = 0$ on $(-\infty , 0]$ and $v (x) = 1$ on $[1,\infty)$. Let $N=X^{\frac{1}{20}}$ and
	$$
	S_d(s)=\sum_{(l,2q)=1}\frac{b(l)\chi_{8d}(l)}{l^s}w_N(l),
	$$
	where 
	$$
	w_N(l) = v \left( \frac{\log (N^2/l)}{\log N} \right).
	$$
	We note for all $\eta>0$ we have that
	$$
	w_N(l) = \frac{1}{2\pi i}\int_{(\eta)} V(z) l^{-z}dz = 1+ \frac{1}{2\pi i}\int_{(-\eta)} V(z) l^{-z}dz,
	$$
	where 
	$$
	V(z) = \frac{1}{z} \int_0^1 v'(x) N^{(2-x)z}dx.
	$$
	Also note that, for all $\eta>0$, 
	$$
	\int_{(\eta/\log X)} \abs{V(z)} \abs{dz} \ll_\eta 1.
	$$
	Let 
	$$
	E_d(s)= L^{*}(s,\pi \otimes \chi_{8d})\mathcal{N}_d(s,k-\tfrac1r)^r-S_d(s)^r,
	$$
	Let
	$$
	\mathcal{J}_d(\sigma) = \int_{-\infty}^{\infty}\abs{L^{*}(\sigma+it,\pi \otimes \chi_{8d})}^{2/r}\abs{\mathcal{N}_d(\sigma+it,k-\tfrac1r)}^2 \abs{W_X(t-\sigma i +\tfrac12i)}^2\,dt,
	$$
	$$
	\mathcal{K}_d(\sigma) = \int_{-\infty}^{\infty}\abs{E_d(\sigma+it)}^{2/r} \abs{W_X(t-\sigma i +\tfrac12i)}^2\,dt,
	$$
	$$
	\mathcal{L}_d(\sigma) = \int_{-\infty}^{\infty}\abs{S_d(\sigma+it)}^{2} \abs{W_X(t-\sigma i +\tfrac12i)}^2\,dt,
	$$
	and 
	$$
	\mathcal{J}(\sigma) = \sumflat_{d \le X} \mathcal{J}_d(\sigma), \quad \mathcal{K}(\sigma) = \sumflat_{d \le X} \mathcal{K}_d(\sigma), \quad
	\mathcal{L}(\sigma) = \sumflat_{d \le X} \mathcal{L}_d(\sigma).
	$$
	Since 
	\begin{align*}
		\abs{S_d(s)}^{2}&=\abs{L^{*}(s,\pi \otimes \chi_{8d})\mathcal{N}_d(s,k-\tfrac1r)^r-E_d(s)}^{2/r}
		\\
		&\ll\abs{L^{*}(s,\pi \otimes \chi_{8d})}^{2/r}\abs{\mathcal{N}_d(s,k-\tfrac1r)}^2+\abs{E_d(s)}^{2/r},
	\end{align*}
	it follows that 
	\begin{equation}\label{eq:I}
		\mathcal{L}(\sigma)\ll \mathcal{J}(\sigma)+\mathcal{K}(\sigma),
	\end{equation}
	and similarly, 
	\begin{equation}\label{eq:II}
		\mathcal{K}(\tfrac12)\ll \mathcal{J}(\tfrac12)+\mathcal{L}(\tfrac12).
	\end{equation}
	
	The following lemma is a variant of a convexity theorem of Gabriel \cite[Theorem~2]{Gabriel}, due to Heath-Brown \cite[Lemma~2]{HeathBrown2010}.
	
	\begin{lem}\label{lem:gabriel}
		Let $f(z)$ and $g(z)$ be holomorphic in the infinite strip $\alpha < \Re(z) < \beta$, and continuous for $\alpha \le \Re(z) \le \beta$. Let $b$ and $c$ be positive real numbers. Suppose that $\abs{f(z)}^b\abs{g(z)}^c$ and $\abs{g(z)}$ tend to zero as $\abs{\Im(z)}\to\infty$, uniformly for $\alpha \le \Re(z) \le \beta$.	Then for $\alpha \le \gamma \le \beta$, we have
		\begin{align*}
			\int_{-\infty}^{\infty} \abs{f(\gamma + it)}^{b} \abs{g(\gamma + it)}^{c} \, dt &\le \left\{ \int_{-\infty}^{\infty} \abs{f(\alpha + it)}^{b} \abs{g(\alpha + it)}^{c} \, dt \right\}^{(\beta - \gamma)/(\beta - \alpha)} 
			\\
			&\qquad\qquad 
			\times
			\left\{ \int_{-\infty}^{\infty} \abs{f(\beta + it)}^{b} \abs{g(\beta + it)}^{c} \, dt \right\}^{(\gamma - \alpha)/(\beta - \alpha)}.
		\end{align*}
	\end{lem}
	
	Applying \Cref{lem:gabriel} we show that $\mathcal{J}_d(\sigma)$ and $\mathcal{K}_d(\sigma)$ satisfy the following convexity estimates.
	
	\begin{lem}\label{lem:JKlem}
		Let $1/2\le \sigma\le 3/4$ and $X>0$ be large. Then,
		$$
		\mathcal{J}_d(\sigma) \ll  X^{\sigma-1/2} \mathcal{J}_d(1/2)^{3/2-\sigma} 
		~~\text{ and }~~
		\mathcal{K}_d(\sigma) \ll X^{-\delta(\sigma-1/2)} \mathcal{K}_d(1/2)^{\frac{2}{3}(2-\sigma)},
		$$
		uniformly in $d$.
	\end{lem}
	\begin{proof} In both cases below we take $g(z)=W_X(-i(z-1/2))$. For $y\ge0$, we have by Plancherel's theorem,
		\begin{equation}\label{eq:Wshift}
			\int_{-\infty}^{\infty}\abs{W_X(t-iy)}^2\,dt=2\pi\int_{-\delta}^{\delta}\abs{\phi_\delta(x)}^2X^{-2xy}\,dx\ll X^{2\delta y}.
		\end{equation}
		Moreover, $W_X(t-iy)$ decays faster than any power of $|t|$, uniformly for $0\le y\le\frac32$. By the functional equation and the Phragm\'en--Lindel\"of principle, we see that $L^*(\sigma+it,\pi\otimes\chi_{8d})$ has at most polynomial growth in $|t|$ for $\frac12\le\sigma\le2$.
		
		To prove the estimate for $\mathcal{J}_d(\sigma)$ in \Cref{lem:JKlem}, we take
		$$
		f(z)= L^{*}(z,\pi \otimes \chi_{8d})\mathcal{N}_d(z,k-1/r)^r,
		$$ 
		with $\gamma=\sigma$, $\alpha=1/2$, $\beta=3/2$, $b=2/r$ and $c=2$, where $1/2\le\sigma\le 3/4$. On the line $\Re z=3/2$ we have by \Cref{lem:proplem2} that
		$$
		|L^*(z,\pi\otimes\chi_{8d})|\ll1
		~~\text{ and }~~ |\mathcal{N}_d(z,k-1/r)|\ll1.
		$$ 
		Hence, 
		\begin{align*}
			\mathcal{J}_d(3/2)
			&\ll \int_{-\infty}^\infty \abs{W_X(t -i)}^2 \,dt
			\\
			&\ll X^{2\delta}\le X.
		\end{align*}
		We conclude from \Cref{lem:gabriel} that
		$$
		\mathcal{J}_d(\sigma)\le\mathcal{J}_d(1/2)^{3/2-\sigma}\mathcal{J}_d(3/2)^{\sigma-1/2} \ll  X^{\sigma-1/2} \mathcal{J}_d(1/2)^{3/2-\sigma} .
		$$
		
		To prove the estimate for $\mathcal{K}_d(\sigma)$ in \Cref{lem:JKlem}, we take 
		$$
		f(z)=E_d(z),
		$$
		with $\gamma=\sigma$, $\alpha=1/2$, $\beta=2$, $b=2/r$ and $c=2$, where $1/2\le\sigma\le 3/4$.  By \Cref{lem:proplem4},
		$$
		L^{*}(s,\pi\otimes\chi_{8d})^{1/r}\mathcal{N}_d(s,k-1/r)
		=\sum_{(l,2q)=1}\frac{b(l)\chi_{8d}(l)}{l^{s}}\qquad(\sigma>1),
		$$
		since $X_1>2q$ forces every $n\in\mathcal{N}$ to be coprime to $2q$. As
		$w_N(l)=1$ for $l\le N$, 
		$$
		S_d(s) = \sum_{\substack{ l\le N \\ (l,2q)=1}} \frac{b(l)\chi_{8d}(l)}{l^s} + \sum_{\substack{ l > N \\ (l,2q)=1}} \frac{b(l)\chi_{8d}(l)}{l^s}w_N(l).
		$$ 
		Thus, $E_d(s)=\sum_{l>N}e_d(l)l^{-s}$ where $\abs{e_d(l)}\ll l^{\theta}$ uniformly in $d$ by \Cref{lem:proplem2}. Hence,
		$$
		\abs{E_d(2+it)}\le\sum_{l>N}\frac{\abs{e_d(l)}}{l^{2}}
		\ \ll \ N^{\theta-1}\le N^{-1/2}.
		$$
		By \eqref{eq:Wshift} with $y=3/2$, and since $N^{1/r}=X^{1/(20r)}=X^{5\delta}$, we have
		\begin{align*}
			\mathcal{K}_d(2) 
			&= \int_{-\infty}^{\infty}\abs{E_d(2+it)}^{2/r} \abs{W_X(t-\tfrac32 i)}^2\,dt
			\\
			&\ll N^{-1/r}X^{3\delta}=X^{-2\delta}.
		\end{align*}
		We conclude from \Cref{lem:gabriel} that
		\begin{align*}
			\mathcal{K}_d(\sigma)&\le\mathcal{K}_d(1/2)^{\frac23(2-\sigma)}\mathcal{K}_d(2)^{\frac23(\sigma-\frac12)}\\
			&\ll X^{-\frac43\delta(\sigma-1/2)} \mathcal{K}_d(1/2)^{\frac{2}{3}(2-\sigma)}\le X^{-\delta(\sigma-1/2)} \mathcal{K}_d(1/2)^{\frac{2}{3}(2-\sigma)} .
		\end{align*}
		
	\end{proof}
	
	Throughout we will denote by $C\ge1$ a constant depending only on $\pi$, not necessarily the same at each occurrence.

	\begin{lem}\label{lem:trunclem}
		We have
		$$
		\norm{ \frac{b(l)}{l^{\sigma+it}} }
		\asymp
		\norm{ \frac{b'(l)}{l^{\sigma+it}} } ,
		$$
		uniformly for $\frac{1}{2} + \frac{1}{\log X} \le \sigma \le 1$.
	\end{lem}
	
	\begin{proof}
		By \Cref{lem:multilem}, we have the following factorization,
		$$
		\norm{ \frac{b(l)}{l^{\sigma+it}} }^2 = \norm{\frac{b(l)}{l^{\sigma+it}}\mathds{1}_{\mathcal{S}_1}(l)}^2
		\times
		\norm{\frac{b(l)}{l^{\sigma+it}}\mathds{1}_{\mathcal{S}_\infty}(l)}^2
		\times
		\prod_{j=2}^{\ell} \norm{ \frac{b(l)}{l^{\sigma+it}}\mathds{1}_{\mathcal{S}_j}(l)}^2,
		$$
		and similarly for $b'$. For $l\in\mathcal{S}_1\cup\mathcal{S}_\infty$, we have $b(l)=b'(l)$. For $l \in \mathcal{S}_j$ we see that
		\begin{align*}
			\abs{b(l)-b'(l)} &\le  \sum_{\substack{mn=l \\ \Omega(n) > K_j }}\abs{\lambda_{1/r}(m)} (k-1/r)^{\Omega(n)} \abs{ g(n) }
			\\
			&\le
			e^{-K_j}\sum_{mn=l}\abs{\lambda_{1/r}(m)} e^{\Omega(n)}(k-1/r)^{\Omega(n)} \abs{ g(n) }.
		\end{align*}
		Thus by \Cref{lem:multilem} and \Cref{lem:proplem2}, for $2\le j \le \ell$,
		$$
		\norm{ \frac{|b(l)-b'(l)|}{l^{\sigma}} \mathds{1}_{\mathcal{S}_j}}^2
		\le e^{-2K_j}\exp\left(C(k+1)^2\sum_{X_{j-1}<p\le X_j}\frac{1+\lambda(p)^2+|B(p)|}{p}+O(1)\right).
		$$
		It follows from \Cref{lem:sqlem}, \Cref{lem:avsqlem} (when $D=4$), and \Cref{lem:proplem3} that
		$$
		\sum_{X_{j-1}<p\le X_j}\frac{1+|B(p)|}{p}\ll \sum_{X_{j-1}<p\le X_j}\frac{\lambda(p)^2}{p}=P_j.
		$$
		Thus, after enlarging $C_0=C_0(\pi,k)$ if necessary, we have
		$$
		\norm{ \frac{|b(l)-b'(l)|}{l^{\sigma}} \mathds{1}_{\mathcal{S}_j}}^2
		\ll e^{-2K_j}\exp\bigl(C(k+1)^2P_j\bigr)
		\ll e^{-K_j}.
		$$
		Also, since $b'(p)=k\lambda(p)$ and $2b'(p^2)=k^2\lambda(p)^2+a(p^2)/r$ for $X_1<p\le X_\ell$, we have
		\begin{align*}
			\norm{ \frac{b'(l)}{l^{\sigma+it}}\mathds{1}_{\mathcal{S}_j}(l)}^2 
			&= 
			\prod_{X_{j-1} < p \le X_j} \left(1+\rho(p)\frac{\abs{b'(p)}^2+2\Re( b'(p^2)p^{-2it})}{p^{2\sigma}} + O\left(\frac{A(p)^2}{p^{2(1-\theta)}}\right) \right)
			\\
			&\ge
			\prod_{X_{j-1} < p \le X_j} \left(1-\frac{\abs{a(p^2)}}{rp} + O\left(\frac{A(p)^2}{p^{2(1-\theta)}}\right) \right)
			\\
			&\ge
			\prod_{X_{j-1} < p \le X_j} \left(1 - \frac{C\bigl(1+\lambda(p)^{2}+|B(p)|\bigr)}{rp} + O\left(\frac{A(p)^2}{p^{2(1-\theta)}}\right) \right) \gg e^{-4 P_j},
		\end{align*}
		in the last step we have used that $r$ is a large integer. Hence, by \eqref{eq:minkowski} and since $K_j\ge8P_j$,
		$$
		\norm{ \frac{b(l)}{l^{\sigma+it}} \mathds{1}_{\mathcal{S}_j}} = \left(1+O(e^{-K_j/4})\right)
		\norm{ \frac{b'(l)}{l^{\sigma+it}} \mathds{1}_{\mathcal{S}_j}}.
		$$
		By \Cref{lem:proplem3},
		$$
		P_j=2(\log_jX-\log_{j+1}X)+O(1)\quad(j\ge3),\qquad P_2=\log_2X-2\log_3X+O(1),
		$$
		so $P_j\asymp\log_jX$ and $K_j\asymp\log_jX$ for $2\le j\le\ell$, since $\log_jX\ge C_0$. As $\log_{j-1}X=e^{\log_jX}$, we have
		$K_{j-1}\ge 2K_j$ for $C_0$ large. Thus
		$$
		\sum_{j=2}^{\ell}e^{-K_j/4}\le 2e^{-K_\ell/4}~~
		\text{ and }~
		\prod_{2\le j\le \ell}\left(1+O(e^{-K_j/4})\right)\asymp 1.
		$$
		Multiplying the estimates for the blocks $\mathcal{S}_j$, and using $b=b'$ on $\mathcal{S}_1\cup\mathcal{S}_\infty$, we obtain \Cref{lem:trunclem}.
	\end{proof}

	\begin{lem}\label{lem:sizelem}
		Uniformly for $\frac12+\frac{1}{\log X}\le\sigma\le1$ we have
		$$
		\norm{ \frac{\abs{b'(l)}}{l^{\sigma}} }^2
		\ \ll\ \left(\frac{1}{2\sigma-1}\right)^{2k^{2}+\frac{C}{r}},
		$$
		and, if in addition $\abs{2\sigma-1+2it}\le1$ then
		$$
		\norm{ \frac{b'(l)}{l^{\sigma +it}} }^2
		\ \asymp\ \left(\frac{1}{2\sigma-1}\right)^{k^{2}}
		\left(\frac{1}{\abs{2\sigma-1+2it}}\right)^{k^{2} + \frac{\epsilon}{r}}
		\ \ll\ \left(\frac{1}{2\sigma-1}\right)^{2k^{2} + \frac{\epsilon}{r}}.
		$$
	\end{lem}
	
	\begin{proof}
		As in the proof of \Cref{lem:trunclem}, by \Cref{lem:multilem} we get
		$$
		\norm{ \frac{b'(l)}{l^{\sigma +it}} }^2
		\asymp\prod_{p}\left(1+\frac{k^2\lambda(p)^2}{p^{2\sigma}}+k^2\Re\frac{\lambda(p)^2}{p^{2\sigma+2it}}+\frac{1}{r}\Re\frac{a(p^2)}{p^{2\sigma+2it}} + O\left(\frac{A(p)^2}{p^{2(1-\theta)}}\right) \right).
		$$
		Hence by \Cref{lem:proplem3},
		$$
		2\log \norm{ \frac{b'(l)}{l^{\sigma +it}} } = k^2\log\left(\frac{1}{2\sigma-1}\right)+\left(k^2+\frac{\epsilon}{r}\right)\log\left(\frac{1}{\abs{2\sigma-1+2it}}\right) +O(1),
		$$
		which gives the second assertion. Similarly,
		$$
		\norm{ \frac{\abs{b'(l)}}{l^{\sigma}} }^2
		\ll\prod_{p}\left(1+\frac{2k^2\lambda(p)^2}{p^{2\sigma}}+\frac{C}{r}\frac{1+\lambda(p)^2+|B(p)|}{p^{2\sigma}} + O\left(\frac{A(p)^2}{p^{2(1-\theta)}}\right) \right).
		$$
		By \Cref{lem:proplem3}, \Cref{lem:avsqlem} (using \Cref{lem:sqlem} when $D=2,3$) and $\sum_pp^{-2\sigma}=\log\frac{1}{2\sigma-1}+O(1)$, the Euler product on the right is
		$$
		\ll (2\sigma-1)^{-2k^2-C/r}.
		$$
		This finishes the proof of \Cref{lem:sizelem}.
		
	\end{proof}
	
	\begin{lem}\label{lem:Llem1}
		Let $X>0$ be large. We have
		$$
		\mathcal{L}(1/2) \ll X(\log X)^{2k^2+\epsilon/r}.
		$$
	\end{lem}
	\begin{proof} We have by \Cref{lem:charlem},
		
		$$
		\sumflat_{d \le X } \abs{S_d(1/2+it) }^2 = \frac{4X}{\pi^2} \norm{ \frac{b(l)w_N(l)}{l^{1/2 +it}} }^2 
		+ O(X^{4/5}).
		$$
		Taking $\eta=1/(\log X)$, we see that
		$$
		\norm{ \frac{b(l)w_N(l)}{l^{1/2 +it}} }^2 =
		\frac{1}{2\pi i}\int_{(\eta)}V(z) \her{ \frac{b(l_1)}{l_1^{1/2 + it + z}} , \frac{b(l_2)w_N(l_2)}{l_2^{1/2 +it}} }  dz.
		$$
		An application of \eqref{eq:cauchy} with $a_1(l) = b(l) l^{-1/2-it-z}$ and $a_2(l) = b(l) l^{-1/2-it} w_N(l)$ shows that
		$$
		\abs{ \her{ \frac{b(l_1)}{l_1^{1/2 + it + z}} , \frac{b(l_2)w_N(l_2)}{l_2^{1/2 +it}} } }
		\le 
		\norm{ \frac{b(l)}{l^{1/2 + it + z}} } \cdot
		\norm{ \frac{b(l)w_N(l)}{l^{1/2 +it}} }.
		$$
		Hence from \Cref{lem:trunclem}, we get 	
		\begin{align*}
			\norm{ \frac{b(l)w_N(l)}{l^{1/2 +it}} }^2
			&\ll
			\int_{-\infty}^{\infty} \abs{V(\eta+iv)} \norm{ \frac{b(l)}{l^{1/2 +\eta + i(t+v)}} }^2 dv
			\\
			&\ll
			\int_{-\infty}^{\infty} \abs{V(\eta+iv)} \norm{ \frac{b'(l)}{l^{1/2 +\eta + i(t+v)}} }^2 dv.
		\end{align*}
		
		By \Cref{lem:sizelem} we have,
		$$
		\norm{ \frac{b'(l)}{l^{1/2 + \eta + i(t+v)}} }^2
		\ll
		\begin{cases}
			(\log X)^{2k^{2}+\frac\epsilon r} & \text{if }\abs{t+v} \le \frac{1}{4},
			\\
			(\log X)^{2k^{2}+\frac{C}{r}} & \text{if }\abs{t+v} > \frac{1}{4}.
		\end{cases}
		$$
		When $\abs{t+v} > \frac{1}{4}$ we have either $\abs{t}>1/8$ or $\abs{v}>1/8$. If $\abs v>\tfrac18$ then
		$$
		\int_{\abs {v} >1/8}\abs{V(\eta+iv)}dv\ll_A(\log X)^{-A};
		$$
		and if $\abs t >\tfrac18$ then
		$$
		\int_{\abs t >1/8}\abs{W_X(t)}^{2}dt\ll_B(\log X)^{-B}.
		$$
		Since, 
		\begin{align*}
			\mathcal{L}(1/2) &\ll X \int_{-\infty}^{\infty} \int_{-\infty}^{\infty} \abs{W_X(t)}^2 \abs{V(\eta+iv)} \norm{ \frac{b'(l)}{l^{1/2 +\eta + i(t+v)}} }^2 dvdt
			\\
			&= X \left( \iint_{\abs{t+v}\le \frac{1}{4}}+\iint_{\abs{t+v}> \frac{1}{4}} \right) \abs{W_X(t)}^2 \abs{V(\eta+iv)} \norm{ \frac{b'(l)}{l^{1/2 + \eta + i(t+v)}} }^2 dvdt
			\\
			& \ll 
			X(\log X)^{2k^{2}+\frac\epsilon r} + X(\log X)^{2k^{2}+\frac{C}{r}-A} + X(\log X)^{2k^{2}+\frac{C}{r}-B}.
		\end{align*}
		\Cref{lem:Llem1} then follows on taking $A$ and $B$ larger than $2k^{2}+\frac{C}{r}+2$.
	\end{proof}
	
	\begin{lem}\label{lem:Llem2}
		Let $X>0$ be large. We have
		$$
		\mathcal{L}(\sigma) \gg X(\sigma-1/2)^{-2k^2-\epsilon/r}
		$$
		uniformly for 
		$$
		\frac{1}{2}+\frac{c_0(\pi,k)}{\log X} \le \sigma\le \frac34		
		$$
		for some $c_0(\pi,k)>0$ large enough.
	\end{lem}
	
	\begin{proof}
		Consider
		$$
		\mathcal{L}(\sigma) \ge \sumflat_{d\le X}\int_{-50r/\log X}^{50r/\log X}\abs{S_d(\sigma+it)}^{2} \abs{ W_X(t-i(\sigma-1/2))}^2\,dt.
		$$
		Noting that for $\abs{t}\le 50r/ (\log X)=(2\delta\log X)^{-1}$,
		\begin{align*}
			\abs{ W_X(t-i(\sigma-1/2)) - W_X(-i(\sigma-1/2)) }
			\le \tfrac12 \abs{ W_X(-i(\sigma-1/2)) }.
		\end{align*}
		Moreover,
		\begin{align*}
			\abs{W_X(t-i(\sigma-1/2))} 
			&\ge  \tfrac12 \abs{ W_X(-i(\sigma-1/2)) } 
			\\
			&\ge  \tfrac12 \abs{ W_X(0) }
			\\
			&\gg (\log X)^{1/2}.
		\end{align*}
		Thus,
		$$
		\mathcal{L}(\sigma) \gg (\log X)\sumflat_{d\le X}\int_{-50r/\log X}^{50r/\log X}\abs{S_d(\sigma+it)}^{2} \,dt.
		$$
		Since we have
		$$
		\sumflat_{d \le X } \abs{S_d(\sigma+it) }^2 = \frac{4X}{\pi^2} \norm{ \frac{b(l)w_N(l)}{l^{\sigma +it}} }^2 + O(X^{4/5}).
		$$ 
		Applying \eqref{eq:minkowski} with $a_1(l) = b(l)w_N(l)l^{-\sigma-it}$ and 
		$$
		a_2(l) = b(l)(w_N(l) - 1)l^{-\sigma -it} = \frac{1}{2\pi i}\int_{(-\eta)} V(z) \frac{b(l)}{l^{\sigma + it + z}}dz
		$$
		for $\eta>0$, we get 
		\begin{equation}\label{eq:Llem2eq}
			\norm{ \frac{b(l)}{l^{\sigma +it}} }
			\le
			\norm{ \frac{b(l)w_N(l)}{l^{\sigma +it}} }  +
			\norm{ a_2(l) } .
		\end{equation}
		Write
		$$
		\norm{ a_2(l) }^2 = \frac{1}{2\pi i}\int_{(-\eta)} V(z)  \her{\frac{b(l_1)}{l_1^{\sigma +it +z}} , a_2(l_2)} \,dz,
		$$
		then it follows from \eqref{eq:cauchy} that
		$$
		\norm{ a_2(l) } \le \frac{1}{2\pi}\int_{(-\eta)} \abs{V(-\eta+iv)}  \norm{\frac{b(l)}{l^{\sigma -\eta +i(t + v)}}} \,dv.
		$$
		We choose $\eta=(\sigma -1/2)/2$, so that $2(\sigma-\eta)-1=\sigma-\frac12$. For the range $\abs{v}\le 1/8$ we have $\abs{ 2(\sigma - \eta)-1 + 2i(t + v)} \le 1$, thus by \Cref{lem:trunclem} and \Cref{lem:sizelem} we get
		$$
		\norm{ \frac{b(l)}{l^{\sigma -\eta + i(t+v)}} }^2 \ll
		(\sigma-1/2)^{-(2k^{2} + \epsilon/r)}.
		$$
		On the line $\Re z=-\eta$ we have $|N^{(2-x)z}|\le N^{-\eta}$, so we get 
		$$
		\int_{(-\eta)}|V(z)||dz|\ll N^{-\eta}=X^{-\frac{\sigma-1/2}{40}}.
		$$ 
		Hence,
		$$
		\int_{\abs{v} \le 1/8} \abs{V(-\eta +iv)} \norm{ \frac{b(l)}{l^{\sigma -\eta + i(t+v)}} } \,dv
		\ll
		X^{-\frac{\sigma-1/2}{40}} (\sigma-1/2)^{-\frac12 (2k^{2} + \epsilon/r)}.
		$$
		For $\abs{v}>1/8$ we have by \Cref{lem:sizelem},
		$$
		\int_{\abs{v} > 1/8} \abs{V(-\eta +iv)} \norm{ \frac{b(l)}{l^{\sigma -\eta + i(t+v)}} } \,dv
		\ll_B 
		(\log X)^{2k^2 + \frac{C}{r} - B}.
		$$
		For $B>2k^2 + \frac{C}{r} + 2$ this is negligible. By \eqref{eq:Llem2eq}, \Cref{lem:trunclem} and \Cref{lem:sizelem} we get
		$$
		(\sigma-1/2)^{-\frac12 (2k^{2} + \epsilon/r)}      
		\ll
		\norm{ \frac{b(l)w_N(l)}{l^{\sigma +it}} }
		+\quad X^{-\frac{\sigma-1/2}{40}} (\sigma-1/2)^{-\frac12 (2k^{2} + \epsilon/r)}.
		$$
		Taking $\sigma \ge \frac{1}{2} + \frac{c}{\log X}$ with $c>0$ large enough, we get 
		$$
		\norm{ \frac{b(l)w_N(l)}{l^{\sigma +it}} }\gg (\sigma-1/2)^{-\frac12 (2k^{2} + \epsilon/r)}  .
		$$
		Hence,
		\begin{align*}
			\mathcal{L}(\sigma) 
			&\ge (\log X)\sumflat_{d\le X}\int_{-50r/\log X}^{50r/\log X}\abs{S_d(\sigma+it)}^{2}\,dt
			\\
			&\gg X(\log X)\int_{-50r/\log X}^{50r/\log X}\norm{ \frac{b(l)w_N(l)}{l^{\sigma +it}} }^2\,dt
			\\
			&\gg X(\sigma -1/2)^{-(2k^2+\epsilon/r)},
		\end{align*}
		uniformly for $\frac{1}{2} +\frac{c}{\log X} \le \sigma \le \frac34$ .

	\end{proof}

	We now deduce \Cref{lem:mainlem*} from \Cref{lem:JKlem}, \Cref{lem:Llem1} and \Cref{lem:Llem2} following the argument of Heath-Brown \cite[page 76]{HeathBrown1981}. 
	
	\begin{proof}[Proof of \Cref{lem:mainlem*}]
		We may assume that 
		$$
		\mathcal{J}(1/2) \le X (\log X)^{2k^2+\epsilon/r} ,
		$$
		for otherwise \Cref{lem:mainlem*} is established. Then by \eqref{eq:II}, \Cref{lem:Llem1}, \Cref{lem:JKlem} and H\"older's inequality over $d$, we get 
		$$
		\mathcal{K}(\sigma)\ll X^{-\delta(\sigma-1/2)} X(\log X)^{2k^2+\epsilon/r}.
		$$
		By \eqref{eq:I}, \Cref{lem:JKlem} and H\"older's inequality, we get 
		$$
		\mathcal{L}(\sigma)\ll X^{2(\sigma-1/2)}\mathcal{J}(1/2)^{3/2-\sigma} + X^{-\delta(\sigma-1/2)} X(\log X)^{2k^2+\epsilon/r}.
		$$
		Choose $\sigma=1/2+c/\log X$ with $c>c_0(\pi,k)$. Using \Cref{lem:Llem2} we get 
		$$
		X(\log X)^{2k^2+\epsilon/r} c^{-2k^2-\epsilon/r} \ll  \bigl(\mathcal{J}(1/2)+1\bigr) e^{2c}  + X(\log X)^{2k^2+\epsilon/r} e^{ -\delta c }.
		$$
		Hence, for $c=c(\pi,k)>0$ sufficiently large, we see that 
		$$
		\mathcal{J}(1/2)\gg X(\log X)^{2k^2+\epsilon/r}.
		$$
		This finishes the proof of \Cref{lem:mainlem*}.
	\end{proof}

	\section*{ACKNOWLEDGMENTS}
	
	Author would like to thank Professor Sanoli Gun and Gaurav Kumar for fruitful discussions and constant encouragement.
	
	\section*{DECLARATION OF GENERATIVE AI USE}
	
	During the preparation of this work, the author used Claude (Anthropic) to assist with the literature survey on automorphic $L$-functions. The author takes full responsibility for the content of this manuscript.

\end{document}